\documentclass[11pt,oneside,english]{amsart}
\usepackage{bbm}
\usepackage{mathrsfs}
\usepackage{amsfonts}
\usepackage{amssymb}
\usepackage[colorlinks]{hyperref}
\usepackage{graphicx}

\makeatletter \theoremstyle{plain}
 \newtheorem{thm}{Theorem}[section]
 
 \newtheorem{prop}[thm]{Proposition}
 
 \numberwithin{equation}{section} %% Comment out for sequentially-numbered
 \numberwithin{figure}{section} %% Comment out for sequentially-numbered
 \theoremstyle{plain}
 \theoremstyle{definition}
 \newtheorem{defn}[thm]{Definition}
 
  \newtheorem{cor}[thm]{Corollary}

\newcommand{\cB}{{{\mathcal B}}}

\newcommand{\cH}{{{\mathcal H}}}
\newcommand{\fH}{{{\mathfrak H}}}

\newcommand{\fS}{{{\mathfrak S}}}

\newcommand{\bH}{{{\bf H}}}

\newcommand{\bx}{{{\mathbf x}}}

\newcommand{\by}{{{\mathbf y}}}

\newcommand{\C}{{{\mathbb C}}}

\newcommand{\R}{{{\mathbb R}}}

\newcommand{\rd}{{{\rm d}}}

\usepackage{babel}

\makeatother

\begin{document}
\title [Bisectors in the Heisenberg group I]
{Bisectors in the Heisenberg group I}

\author{Gaoshun Gou  \& Yueping Jiang \& Ioannis D. Platis}
\address {School of science,
Chongqing-University of Posts and Telecommunications,
Chongqing {\rm 400065},
P. R. China.}
\email{gaoshungou@hnu.edu.cn}

\address {Department of Mathematics,
Hunan-University,
Changsha {\rm 410082},
P. R. China.}
\email{ypjiang@hnu.edu.cn}
\address{Department of Mathematics and Applied Mathematics,
University of Crete,
University Campus,
GR-70013 Heraklion Crete,
 Greece.}
\email{jplatis@math.uoc.gr}


\begin{abstract}
We show that metric bisectors with respect to the Kor\'anyi metric in the Heisenberg group are spinal spheres and vice versa. We also calculate explicitly their horizontal mean curvature.
\end{abstract}

\thanks{Part of this work has been carried out while IDP was visiting Hunan University, Changsha, PRC. Hospitality is gratefully appreciated. This work is supported by NSFC(No.11631010), NSFC(No.11901061) and cstc2021jcyj-msxm1883.}

\maketitle

\section{Introduction}
A {\it metric bisector} in a metric space $(X,d)$ is the subset $B(x_1,x_2)$ of points $x_1\neq x_2$ of $X$ that are equidistant from both $x_1$ and $x_2$:
$$
B(x_1,x_2)=\{x\in X\;|\;d(x_1,x)=d(x_2,x)\}.
$$
Metric bisectors enjoy the following property: if $f:X\to X$ is a similarity of $X$, that is, a mapping satisfying a relation of the form $d(f(x),f(y))=K_fd(x,y)$ for every $x,y\in X$, where $K_f$ is a positive constant depending only on $f$, then the $f$-image of any bisector is again a bisector. In general, bisectors can be quite complicated objects in an arbitrary metric space $X$.

The same holds for Riemannian manifolds $(M,g)$ with the metric $d_g$ induced by the Riemannian tensor $g$. The most tractable example of a bisector in a Riemannian manifold is of course that of $M=\R^n$, $g=\sum_{i=1}^ndx_i^2$ and $d(\bx,\by)=\|\bx-\by\|$ for each $\bx=(x_1,\dots,x_n)$ and $\by=(y_1,\dots,y_n)$ in $\R^n$. Here, $\|\cdot\|$ is the usual Euclidean norm. Then it follows at once that $B(\bx_1,\bx_2)$ is the hyperplane comprising of $\bx\in\R^n$ such that
$2\bx\cdot(\bx_1-\bx_2)=\|\bx_1\|^2-\|\bx_2\|^2$. We stress that due to invariance by similarities we would have chosen the points ${\bf 0}=(0,\dots,0)$ and ${\bf 1}=(1,0,\dots,0)$. The bisector of this points is the hyperplane $x_1=1/2$ and then we would have concluded that all bisectors are hyperplanes since all images of $x_1=1/2$ by Euclidean similarities are hyperplanes.

In this paper we study bisectors of the Heisenberg group $\fH$ endowed with the Kor\'anyi metric ${\rm d}_K$.
The Heisenberg group $\fH$ is the set $\mathbb{C}\times\mathbb{R}$ with multiplication $*$ given by
$$
(z,t)*(w,s)=(z+w,t+s+2\Im(z\overline{w})),
$$
for every $(z,t)$ and $(w,s)$ in $\fH$. The Kor\'anyi metric is then defined by
$$
{\rm d}_K(p,q)=\|p*q^{-1}\|,
$$
for each $p,q\in\fH$. Here $\|p\|_K=(|z|^4+t^2)^{1/4}$ for each $p=(z,t)\in\fH$, see Section \ref{sec:heisen} for details.

We note that this point that the Heisenberg group is a simple sub-Riemmanian manifold. Any such manifold  is naturally equipped with the  Carnot-Carath\'eodory metric $\mathrm{d}_{cc}$, see Section \ref{sec:heisen} and the references within. The metric $\mathrm{d}_{cc}$ is related to the metric ${\rm d}_K$ in quite few ways, for instance, among others we stress here that  both metrics have the same isometry and similarity groups.  However, bisectors with respect to $\mathrm{d}_{cc}$ are generally not the same objects as bisoctors with respect to ${\rm d}_K$; details about the study of those objects will appear elsewhere.

In this paper, we are dealing with bisectors with respect to the Kor\'anyi metric.
Let $p_1,p_2\in\fH$ be two distinct points and the Kor\'anyi bisector
\begin{equation*}
\cB(p_1,p_2)=\{p\in\fH\;|\;{\rm d}_K(p_1,p)={\rm d}_K(p_2,p)\}.
\end{equation*}
We may normalize so that $p_i$, $i=1,2$ lie either in the same finite or in the same infinite $\C$-circle, see Section \ref{sec-bisec} for details. Our main theorem is then the following.
\begin{thm}\label{thm-KB}
A Kor\'anyi bisector is a spinal sphere. Conversely, every spinal sphere is a Kor\'anyi bisector.
\end{thm}
Spinal spheres are the boundaries of bisectors of complex hyperbolic plane with respect to the Bergman metric. Such bisectors  have been used to construct  Dirichlet fundamental domains or Ford fundamental domains of a discrete subgroup of $\rm{PU}(n,1)$.
In particular, Parker and Will used  isometric spheres to construct  Ford fundamental domains in \cite{PW2017}.

Finally, as far as it concerns the horizontal geometry of Kor\'anyi bisectors/spinal spheres we prove that if $p_1$, $p_2$ lie on an infinite $\C$-circle, then the Kor\'anyi bisector is a horizontal minimal surface (see Proposition \ref{prop-min}), that is, its horizontal mean curvature vanishes everywhere. However, this is not the case if $p_1$, $p_2$ lie on an infinite $\C$-circle. The surface tends to be horizontally minimal away from its characteristic locus (see Proposition \ref{prop-notmin}).

This paper is organized as follows. In Section \ref{sec-prel}, we shall review some preliminaries about the complex hyperbolic plane and the bisectors with respect to the Bergman metric, as well as some basic knowledge for the Heisenberg group and the Kor\'anyi metric. In Section \ref{sec-bisec}, we prove Theorem \ref{thm-KB} and in Section \ref{sec-horcurv} we prove Proposition \ref{prop-notmin}.

\section{Preliminaries}\label{sec-prel}

\subsection{Complex hyperbolic plane $\bH^2_\C$ and its bisectors}
Let $\mathbb{C}^{2,1}$ be $\C^3$ equipped with a non degenerate,  Hermitian form $\left\langle\cdot,\cdot\right\rangle$ of signature $(2,1)$: If
$
\mathbf{z}=(z_1,z_2,z_3)^T\quad\text{and}\quad \mathbf{w}=(w_1,w_2,w_3)^T,$
 then
$$
\left\langle \mathbf{z,w}\right\rangle=z_1\overline{w_3}+z_2\overline{w_2}+z_3\overline{w_1}=\mathbf{w}^*H{\mathbf z},
$$
where $$
H=\left[
  \begin{array}{ccc}
    0 & 0& 1 \\
    0 & 1 & 0 \\
    1 & 0 & 0 \\
  \end{array}
\right].
$$
We consider the subsets
 \begin{align*}
      V_-= & \left.\left\{\mathbf{z}\in\mathbb{C}^{2,1}\right|\left\langle \mathbf{z},\mathbf{z}\right\rangle<0\right\} ,\\
      V_0= & \left.\left\{\mathbf{z}\in\mathbb{C}^{2,1}\right|\left\langle \mathbf{z},\mathbf{z}\right\rangle=0\right\} ,\\
      V_+= & \left.\left\{\mathbf{z}\in\mathbb{C}^{2,1}\right|\left\langle \mathbf{z},\mathbf{z}\right\rangle>0\right\}
    \end{align*}
and let also
    $$
\mathbb{P}:\mathbb{C}^{2,1}\ni\left[
             \begin{array}{c}
               z_1 \\
               z_2 \\
               z_3 \\
             \end{array}
           \right]
\mapsto\left[
         \begin{array}{c}
           z_1/z_3 \\
           z_2/z_3 \\
         \end{array}
       \right]\in\mathbb{C}^2.
$$
\begin{defn}
The complex hyperbolic plane $\mathbf{H}_{\mathbb{C}}^2$ is $\mathbb{P}(V_-)$ and its boundary $\partial\mathbf{H}_{\mathbb{C}}^2$ is $\mathbb{P}(V_0)$.
\end{defn}
The standard model for complex hyperbolic plane we use here is
the {\it Siegel domain model}:
$$
\mathbf{H}_{\mathbb{C}}^2=\{(z_1,z_2)\in\mathbb{C}^2\;|\;2\Re(z_1)+|z_2|^2<0\}.
$$
Let $(z_1,z_2)\in\mathbb{C}^2$. The {\it standard lift} of $z$ is
$
\mathbf{z}=(z_1,z_2,1)^T.$
In particular, the standard lift of $\infty$ is $(1,0,0)^T$. %We mention at this point that
%{\it everything which we shall define from now on is independent of the choice of lifts.}
The {\it Bergman metric} of $\mathbf{H}_{\mathbb{C}}^2$ in terms of the distance function $\rho(\cdot,\cdot)$ is given by
$$
\cosh^2\left(\frac{\rho(z,w)}{2}\right)=\frac{\langle \mathbf{z},\mathbf{w}\rangle\langle\mathbf{w},\mathbf{z}\rangle}{\langle\mathbf{z},\mathbf{z}\rangle\langle\mathbf{w},\mathbf{w}\rangle},
$$
and this definition is independent of the choice of lifts. The following hold:
\begin{itemize}
\item The holomorphic sectional curvature is $-1$.
\item The sectional curvature is pinched between $-1$ and $-1/4$.
\item The group of holomorphic isometries is ${\rm PU}(2,1)$. The group ${\rm SU}(2,1)$, a three-fold covering of ${\rm PU}(2,1)$, is also used.
\end{itemize}
There are two types of geodesic submanifolds (of dimension 2):
First, we have {\it complex lines ($\mathbb{C}$-lines)}:  Let $z,w\in\overline{\mathbf{H}_{\mathbb{C}}^2}$ and let
$$
\mathbf{C}(\mathbf{z},\mathbf{w})={\rm span}_\C(\mathbf{z},\mathbf{w}),
$$
with $\mathbf{z},\mathbf{w}$ being lifts of $z,w$, respectively.
The $\mathbb{C}$-line $\mathcal{C}(z,w)$ is the complex projection $\mathbb{P}(\mathbf{C}(\mathbf{z},\mathbf{w}))$ of $\mathbf{C}(\mathbf{z},\mathbf{w})$. $\C$-lines are all isometric to
$$
\mathbf{H}_{\mathbb{C}}^1=\{z\in\mathbb{C}\;|\;\Re(z)<0\}.
$$
Second, we have {\it Lagrangian planes $\mathcal{R}$ (or, $\mathbb{R}$-planes)}: Those are  characterised by $\langle\mathbf{z},\mathbf{w}\rangle\in\mathbb{R}$ for all $z,w\in\mathcal{R}$ and are all isometric to
$$
\mathbf{H}_{\mathbb{R}}^2=\{(z_1,z_2)\in\mathbf{H}_{\mathbb{C}}^2\;|\;\Im(z_1)=\Im(z_2)=0\}.
$$

\subsubsection{$\bH^2_\C$-bisectors.}
 In contrast to the real hyperbolic space case, there are no geodesic submanifolds of dimension $3$. Bisectors are three dimensional submanifolds which are pretty close to be geodesic.
\begin{defn}
Let $z,w\in\mathbf{H}_{\mathbb{C}}^2$ be two distinct points. The {\it bisector} $\cB_\rho(z,w)$ of  $z$ and $w$ is
$$
\cB_\rho(z,w)=\left.\left\{x\in\mathbf{H}_{\mathbb{C}}^2\right|\rho(x,z)=\rho(x,w)\right\},
$$
where $\rho$ is the distance defined by the Bergman metric.
\end{defn}
The following are standard features of a bisector $\cB_\rho(z,w)$:
\begin{itemize}
 \item The {\it complex spine} $\Sigma$ of $\cB_\rho(z,w)$ is the complex geodesic $\mathcal{C}(z,w)$.
\item The {\it spine} $\sigma$ of $B_\rho(z,w)$ is $\cB_\rho(z,w)\cap\Sigma$, which is the geodesic corresponding to $\Sigma$.
\item The endpoints of the spine $\sigma$ are the {\it vertices} of the bisector and they determine it completely.
\end{itemize}
A bisector is foliated in two distinguished manners which are described in the following theorems, see \cite{G}.
\begin{thm}
{\bf Slice decomposition.} Let $\cB$ be a bisector, $\Sigma$ its complex spine and $\sigma$ its spine. Then
$$
\cB=\Pi_\Sigma^{-1}(\sigma)=\bigcup_{p\in\sigma}\Pi_\Sigma(p),
$$
where $\Pi_\Sigma:\bH^2_\C\to\Sigma$ is the orthogonal projection to $\Sigma$.
\end{thm}
\begin{thm}
{\bf Meridianal decomposition.} Let $\sigma$ be a geodesic in $\bH^2_\C$. Then the bisector $\cB$ which has $\sigma$ as its spine, is the union of all Lagrangian planes containing $\sigma$.
\end{thm}

The group of holomorphic isometries ${\rm PU}(2,1)$ of $\bH^2_\C$  acts transitively on bisectors. Therefore we have:
\begin{cor}
All bisectors are isometric to the bisector $\cB_0$ whose spine is $\sigma_0=(0,\infty)$.
\end{cor}
We will also consider for further use the bisector $\cB_1$ whose spine is $\sigma_1=((-1,0),(1,0))$.
\subsubsection{Spinal spheres}
\begin{defn}
A spinal sphere $\fS$ is the boundary of a bisector $\cB$ in $\partial\bH^2_\C$.
\end{defn}
The following hold:
\begin{itemize}
\item A spinal sphere is fully determined by its vertices.
\item ${\rm PU}(2,1)$ acts transitively on spinal spheres.
\item Each spinal sphere is the image of $\fS_0=\partial\cB_0=\C$ via an element of ${\rm PU}(2,1)$.
\end{itemize}
We shall also denote by $\fS_1$ the spinal sphere $\partial\cB_1$. This is the hypersurface with equation
\begin{equation*}\label{eq-fin}
f(x,y,t)=x(x^2+y^2+1)-yt=0.
\end{equation*}
\subsection{Heisenberg group}\label{sec:heisen}
The boundary $\partial\bH^2_\C\setminus\{\infty\}$ is in bijection with the Heisenberg group $\fH$; this  is the set $\mathbb{C}\times\mathbb{R}$ with multiplication $*$ given by
$$
(z,t)*(w,s)=(z+w,t+s+2\Im(z\overline{w})),
$$
for every $(z,t)$ and $(w,s)$ in $\fH$. There are two natural (left invariant) metrics defined in $\fH$. First, we have the {\it Kor\'anyi-Cygan metric} given by
$$
\mathrm{d}_{K}((z,t),(w,s))=|(z,t)^{-1}*(w,s)|_{K},
$$
where $|\cdot|_K$ is the Kor\'anyi gauge given by
$$
|(z,t)|_{K}=||z|^2-it|^{1/2},
$$
for each $(z,t)\in\fH$.

The similarity group ${\rm Sim}(\fH)$ of $\fH$ with respect to the Kor\'anyi metric %and Carnot-Carath\'eodory metrics,
comprises the following transformations:
\begin{enumerate}
\item Left translations $L_p$, $p\in\fH$, defined by
$$
L_p(q)=p*q,
$$
for each $q\in\fH$.
\item Rotations around the vertical axis $R_\theta$, $\theta\in\mathbb{R}$, defined by
$
R_\theta(z,t)=(e^{i\theta}z,t),
$
for each $(z,t)\in\fH$.
\item Dilations $D_\delta$, $\delta>0$, defined by
$
D_\delta(z,t)=(\delta z,\delta^2 t),
$
for each $q\in\fH$.
\item Conjugation $j$, defined by
$
j(z,t)=(\overline{z},-t),
$
for each $q\in\fH$.
\end{enumerate}

Left translations, rotations and conjugation are the isometry group of $\fH$ for $\rd_K$. The similarity group ${\rm Sim}(\fH)$ may be viewed as the isotropy subgroup of $\infty$ in ${\rm SU}(2,1)$, %denoted by ${\rm Iso}_{\infty}({\mathbf{H}_{\mathbb{C}}^2})$,
 see \cite{P}.

The following holds, see \cite[Proposition 2.6]{PW2017} or \cite[ Proposition 3.1]{PW2015}:

 \begin{prop}\label{prop-doubly-tans}
 The similarity group ${\rm Sim}(\fH)$ acts doubly transitively on the Heisenberg group.
 \end{prop}
For clarification purposes, we describe in brief the second metric although it is not of our interest in the present paper. %The second metric has to do with the sub-Riemannian structure of $\fH$.
The Heisenberg group $\fH$ is a 2-step nilpotent Lie group; we consider the basis for the left invariant vector fields of $\fH$ comprising
$$
X=\frac{\partial}{\partial x}+2y\frac{\partial}{\partial t},\quad
Y=\frac{\partial}{\partial y}-2x\frac{\partial}{\partial t},\quad
T=\frac{\partial}{\partial t}.
$$
Denote by $\mathfrak{h}$ the Lie algebra of $\fH$. There exists a decomposition:
$
\mathfrak{h}=V^1\oplus V^2,
$
where
$$
V^1={\rm span}_{\mathbb{R}}\left\{X,Y\right\},\quad V^2={\rm span}_{\mathbb{R}}\left\{T\right\}.
$$
The contact structure of $\fH$ is induced by the 1-form
$$
\omega={\rm d}t+2(x{\rm d}y-y{\rm d}x)={\rm d}t+2\Im(\overline{z}{\rm d}z),
$$
where $z=x+iy$. By the contact version of Darboux's Theorem, $\omega$ is the unique 1-form  such that $X,\,Y\in\ker\omega,\,\omega(T)=1$.
For each point $p\in \fH$, $V^1_p=\cH_p(\fH)$ is the {\it horizontal tangent space of $\fH$ at $p$}.
On the other hand, consider the relations
$$(X,X)_h=(Y,Y)_h=1,\;
( X,Y)_h=( Y,X)_h=0.
$$
From these relations we obtain the  sub-Riemannian metric $(\cdot,\cdot)_h$ in $\fH$; its induced norm shall be denoted by $|\cdot|_h$.
A smooth curve $\gamma:[a,b]\to\fH$  with
$$
\gamma(s)=(z(s),t(s))\in\mathbb{C}\times\mathbb{R},
$$ %$\gamma_h(\tau)=x(\tau)+iy(\tau)$, $\gamma_3(\tau)=t(\tau)$,
is called a {\it horizontal curve} if $\dot{\gamma}\in \cH_{\gamma(s)}(\fH)$ for all $s\in[a,b]$. Equivalently,
$$
\dot{t}(s)=-2\Im\left(\overline{z(s)}\dot{z}(s)\right),
$$
for $s\in[a,b]$.
The {\it horizontal length} of a smooth rectifiable curve $\gamma$ with respect to $|\cdot|_h$ is given by
$$
\ell_h(\gamma)=\int_a^b|\dot{\gamma}_h(s)|_h{\rm d}s=\int_a^b\left[\left(\dot{\gamma}(s),X_{\gamma(s)}\right)_h^2+
\left(\dot{\gamma}(s),Y_{\gamma(s)}\right)_h^2\right]^{1/2}{\rm d}s=\int_a^b|\dot z(s)|ds.
$$
The {\it Carnot-Carath\'eodory} distance $\mathrm{d}_{cc}(p,q)$ between any two points $p,q\in\fH$ is then the infimum of horizontal lengths of all horizontal curves joining $p,q$. We note the following:
\begin{itemize}
\item A  neat way to write down explicitly the distance formula may be found in \cite{CCG}.
\item There is a relation between the Kor\'anyi gauge and the Carnot-Carath\'eodory norm, see Proposition 2.1 in \cite{KN}.
\item The metrics $\mathrm{d}_{K}$ and $\mathrm{d}_{cc}$ are bi-Lipschitz equivalent; however, note  that $\mathrm{d}_{cc}$ is a path metric whereas $\mathrm{d}_{K}$ is not.
\item Both $\mathrm{d}_{K}$ and $\mathrm{d}_{cc}$ have the same isometry and similarity groups.
\end{itemize}

\section{Bisectors in the Heisenberg group}\label{sec-bisec}
Suppose that $\rd_K$ is the Kor\'anyi-Cygan in the Heisenberg group $\fH$. Let $p_1,p_2\in\fH$ be two distinct points and the Kor\'anyi bisector be
\begin{equation*}
\cB(p_1,p_2)=\{p\in\fH\;|\;{\rm d}_K(p_1,p)={\rm d}_K(p_2,p)\}.
\end{equation*}
From the properties of ${\rm d}_K$ it follows immediately that
the image of a Kor\'anyi bisector under a Heisenberg similarity is a Kor\'anyi bisector. Now for any two points $p_1$ and $p_2$ in $\fH$ there is always a unique $\C$-circle passing through $p_1$ and $p_2$ (see Theorem 4.3.5 in \cite{G}).
Heisenberg similarities map finite $\C$-circles to finite $\C$-circles and infinite $\C$-circles to infinite $\C$-circles. If $\cB(p_1,p_2)$ is a Kor\'anyi bisector, by Proposition \ref{prop-doubly-tans}, we may always normalize so that
\begin{enumerate}
\item $p_1=(0,-1)$ and $p_2=(0,1)$ in the case where $p_1,p_2$ lie in the same infinite $\C$-circle. We denote by $\cB^0$ the bisector $\cB((0,-1),(0,1))$.
\item $p_1=(-1,0)$ and $p_2=(1,0)$ in the case where $p_1,p_2$ lie in the same infinite $\C$-circle. We denote by $\cB^1$ the bisector $\cB((-1,0),(1,0))$.
\end{enumerate}
\subsection{Kor\'anyi bisectors}
We are now set to prove the Theorem \ref{thm-KB}:
\begin{center}
\includegraphics[scale=0.3]{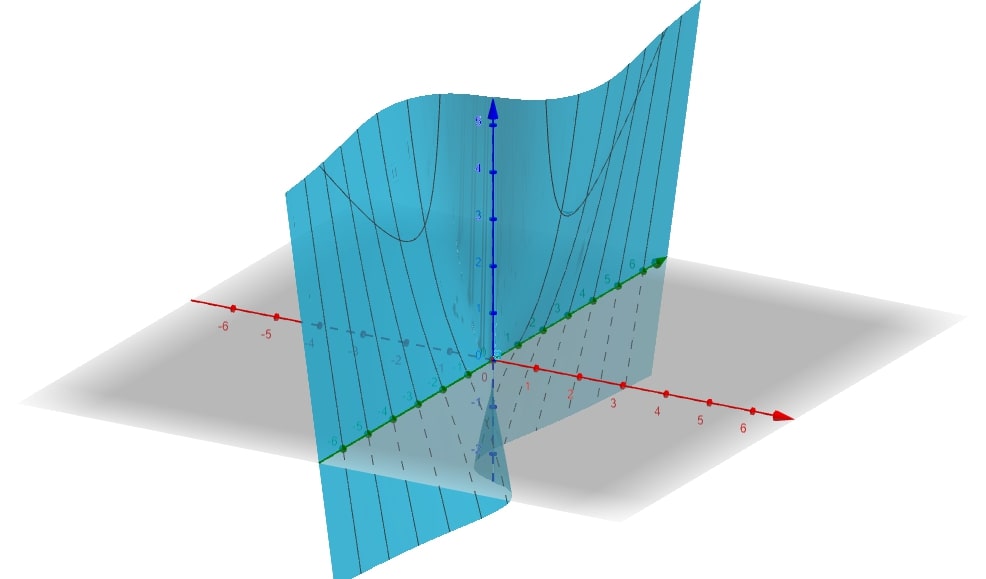}
\begin{figure}[!h]
\centering{Kor\'anyi bisector, finite $\C$-circle case: The surface $f(x,y,t)=x(x^2+y^2+1)-yt=0.$
}
\end{figure}
\end{center}
\begin{proof}
Let $\cB_K(p_1,p_2)$ be a Kor\'anyi bisector. It suffices to consider the cases where this is  $\cB^0_K$ and $\cB^1_K$, as above.

In the case of  $\cB^0_K$, the equation $${\rm d}_{K}((0,-1),p)={\rm d}_{K}((0,1),p)$$ is just
$$
|z|^4+(t+1)^2=|z|^4+(t-1)^2
$$
from where it follows that
$$
\cB^0_K=\{(z,t)\in\fH\;|\;t=0\}=\C.
$$
The complex plane $\C$ is the spinal sphere $\mathfrak{S}_0$ of the complex hyperbolic bisector given by $\Im(z_2)=0$, see also Example 5.1.7 in \cite{G}.

In the case of $\cB^1_K$, the equation
 $$
 {\rm d}_{K}((-1,0),p)={\rm d}_{K}((1,0),p),
 $$
 is %equivalent to $K(z+1,t-y)=K(z-1,t+y)$
 $$
 |z+1|^4+(t-y)^2=|z-1|^4+(t+y)^2.
$$
After short calculations we obtain
the hypersurface
\begin{equation}\label{eq-fin}
f(x,y,t)=x(x^2+y^2+1)-yt=0,
\end{equation}
which is also the spinal sphere,
%compare to Exercise 5.1.4 of \cite{G}. We denote this spinal sphere by
$\mathfrak{S}_1$.

Now conversely, any spinal sphere may be mapped to one of  $\mathfrak{S}_0$ or $\mathfrak{S}_1$ under an element of ${\rm Sim}(\fH)$. The proof is thus complete.
\end{proof}

\subsection{Horizontal geometry of Kor\'anyi bisectors}\label{sec-horcurv}
There is a major distinction in the horizontal geometry of Kor\'anyi bisectors. Before we state our result, we review some basic features of horizontal geometry of hypersurfaces in $\fH$.
Let $F:\fH\to \R$ be a $\mathcal{C}^2$ map and consider the hypersurface $\mathcal{S}$ in $\fH$ defined by the equation $F(x,y,t)=0$. The {\it horizontal normal to} $\mathcal{S}$ is the vector field
$$
N^h_\mathcal{S}=XF\cdot X+YF\cdot Y.
$$
The characteristic locus of $\mathcal{S}$ is the set
$$
C(\mathcal{S})=\{p\in\mathcal{S}\;|\;X_p(F)=Y_p(F)=0\}.
$$
The {\it unit horizontal normal to} $\mathcal{S}$ is
$$
n^h_\mathcal{S}=\frac{N^h_\mathcal{S}}{|N^h_\mathcal{S}|_h},\quad |N^h_\mathcal{S}|_h=\left[(XF)^2+(YF)^2\right]^{1/2}.
$$
Set $n^h_\mathcal{S}=n_1\cdot X+n_2\cdot Y$.
The {\it horizontal mean curvature of} $\mathcal{S}$  is then defined as
\begin{equation}\label{eq-Hh}
2H^h=
%{\rm tr}(W^h)=
X(n_1)+Y(n_2).%,
\end{equation}
Straightforward calculations deduce
\begin{eqnarray*}
&&
X(n_1)=\frac{XXF\cdot (YF)^2-XF\cdot YF\cdot XYF}{\left[(XF)^2+(YF)^2\right]^{3/2}},\\
&&
Y(n_1)=\frac{YXF\cdot (YF)^2-XF\cdot YF\cdot YYF}{\left[(XF)^2+(YF)^2\right]^{3/2}},\\
&&
X(n_2)=\frac{XYF\cdot (XF)^2-XF\cdot YF\cdot XXF}{\left[(XF)^2+(YF)^2\right]^{3/2}},\\
&&
Y(n_2)=\frac{YYF\cdot (XF)^2-XF\cdot YF\cdot YXF}{\left[(XF)^2+(YF)^2\right]^{3/2}}.
\end{eqnarray*}
The above relations show that Eq. (\ref{eq-Hh}) also reads as
\begin{equation}\label{eq-Hexp}
2H^h=\frac{(YF)^2\cdot XXF+(XF)^2\cdot YYF- XF\cdot YF\cdot(XYF+YXF)}{\left[(XF)^2+(YF)^2\right]^{3/2}}.
\end{equation}
Horizontal mean curvature is invariant under Heisenberg similarities. The surface $\mathcal{S}$ is horizontally minimal if $H^h(\mathcal{S})\equiv 0$.
\begin{prop}\label{prop-min}
  In the case where the points lie on an infinite $\C$-circle, a Kor\'anyi bisector is a horizontally minimal surface.
 \end{prop}
 \begin{proof}
The complex plane defined by $F(x,y,t)=t=0$ is well known to be horizontally minimal. For clarity, we carry out the details: We have
$$
XF=2y,\quad YF=-2x,
$$
therefore the characteristic locus comprises the single point $(0,0,0)$. Away from this point,
$$
n^h=\frac{y\cdot X-x\cdot Y}{\sqrt{x^2+y^2}}.
$$
Hence
\begin{eqnarray*}
2H^h&=&X(y/\sqrt{x^2+y^2})-Y(x/\sqrt{x^2+y^2})\\
&=&\partial_x(y/\sqrt{x^2+y^2})-\partial_y(x/\sqrt{x^2+y^2})\\
&=&0.
\end{eqnarray*}
 \end{proof}
This is not the case when the points defining the Kor\'anyi bisector lie in a finite $\C$-circle:
\begin{prop}\label{prop-notmin}
The horizontal mean curvature of the spinal sphere $\mathfrak{S}_1$ diverge to infinity near the characteristic points $(0,\pm 1,0)$ and tends to 0 away from those points.
\end{prop}
\begin{proof}
 $\mathfrak{S}_1$ is the hypersurface given by $$f(x,y,t)=x(x^2+y^2+1)-yt=0.$$
The partial derivatives of $f$ are
\begin{eqnarray*}
&&
f_x=3x^2+y^2+1,\\
&&
f_y=2xy-t,\\
&&
f_t=-y.
\end{eqnarray*}
There are no singular points here, thus $\mathfrak{S}_1$ is a $\mathcal{C}^2$ hypersurface. Now,
\begin{eqnarray*}
&&
Xf=3x^2-y^2+1,\\
&&
Yf=4xy-t.
\end{eqnarray*}
Therefore the characteristic locus is the set of points belonging to both the curve defined by the equations
$$
y^2-3x^2=1,\quad t=4xy,
$$
that is, the intersection of a hyperbolic cylindre and a saddle surface, and to the surface $f(x,y,t)=0$. Plugging in the former two equations in the latter, we have
$$
x(x^2+3x^2+1)-4xy^2=0\implies x=0\;\text{or}\; 4x^2-4y^2=1.
$$
If $x=0$, then $t=0$ and $y^2=1$ so we obtain the points $(0,\pm 1,0)$. If $4x^2-4y^2=1$, then this together with $y^2-3x^2=1$ gives $-8x^2=5$ which is absurd. We conclude that the characteristic locus of the surface comprises the two points $(0,\pm 1,0)$.

As now for the second derivatives, we have
\begin{eqnarray*}
&&
XXf=6x,\quad YYf=6x,\\
&&
XYf=2y,\quad YXf=-2y.
\end{eqnarray*}
Threfore by formula (\ref{eq-Hexp}) we immediately obtain
$$
2H^h=\frac{XXF}{\left[(XF)^2+(YF)^2\right]^{1/2}}=\frac{6x}{\left[(3x^2-y^2+1)^2+(4xy-t)^2\right]^{1/2}}.
$$
The only points $(x,y,t)$ on the surface with $y=0$ are points of the form $(0,0,t)$. At those points $H^h=0$. When $y\neq 0$ we obtain from $f(x,y,t)=0$ that
$$
t=\frac{x}{y}(x^2+y^2+1).
$$
In this manner $H^h$ becomes a function $H^h=H^h(x,y)$ with
$$
H^h(x,y)=\frac{3xy}{\left[y^2(3x^2-y^2+1)^2+x^2(3y^2-x^2-1)^2\right]^{1/2}}.
$$
\begin{center}
\includegraphics[scale=0.3]{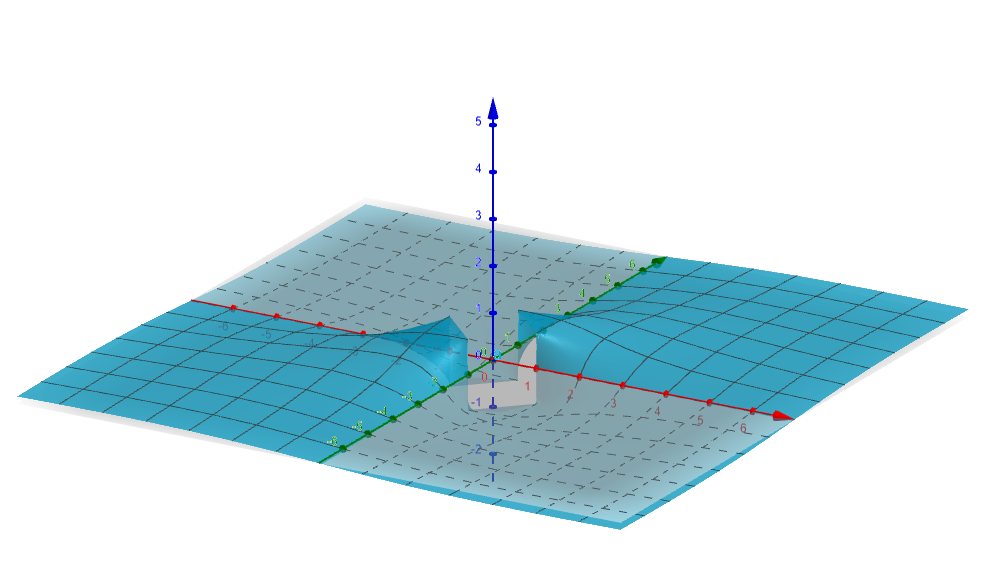}
\begin{figure}[!h]
\centering{The horizontal mean curvature of the surface $f=0.$
}
\end{figure}
\end{center}
 It is now straightforward to show (see also the figure) that the curvature tends to zero away from the critical points and it is bounded near them.
\end{proof}

% alteratively, bibliographies prepared with BibTeX can be included by
% means of the following commands
%\bibliographystyle{srtnumbered}
%\bibliography{mybib}

\end{document}